\documentclass[10pt]{amsart}

\usepackage{tikz}
\usepackage{amsfonts}
\usepackage{graphicx}
\usepackage{amssymb}
\usepackage{amsmath}
\usepackage{amsthm}
\usepackage{color}
\usepackage{graphicx,mathrsfs,comment}
\usepackage{hyperref,url}
\usepackage{mathrsfs}
\usepackage[shortlabels]{enumitem}
\usepackage{esint}
\usepackage{fancyhdr}

\definecolor{plum}{rgb}{.5,0,1}

\theoremstyle{plain}
\newtheorem{theorem}{Theorem}

\newtheorem{lemma}{Lemma}
\newtheorem{proposition}{Proposition}

\theoremstyle{definition}
\newtheorem{definition}{Definition}

\theoremstyle{remark}
\newtheorem*{remark}{Remark}

\newcommand{\dist}{\textup{dist}}
\newcommand{\N}{\mathbb{N}}

\newcommand{\R}{\mathbb{R}}

\newcommand{\e}{\epsilon}

\newcommand{\lesim}{\lesssim}

\newcommand{\om}{\omega}

\newcommand{\cP}{\mathcal{P}}

\newcommand{\Id}{{\bf 1}}

\newcommand{\wh}{\widehat}

\newcommand{\si}{\sigma}
\newcommand{\Si}{\Sigma}
\newcommand{\supp}{\textup{supp}}

\newcommand{\De}{\Delta}
\newcommand{\ZS}{\mathbb S}
\newcommand{\Ga}{\Gamma}
\newcommand{\ga}{\gamma}
\newcommand{\cC}{\mathcal C}
\newcommand{\be}{\beta}

\newcommand{\bS}{\textbf{S}}
\newcommand{\cG}{\mathcal{G}}
\newcommand{\bU}{\textbf{U}}

\makeatletter
\@namedef{subjclassname@2020}{%
  \textup{2020} Mathematics Subject Classification}
\makeatother

\title{Small cap square function estimates}
\date{}

\author{Shengwen Gan} \address{Shengwen Gan\\  Deparment of Mathematics\\ University of Wisconsin-Madison\\ Madison, WI-53706, USA}\email{sgan7@wisc.edu}

\subjclass[2020]{42B10}

\begin{document}

\begin{abstract}



We introduce and prove small cap square function estimates for the unit parabola and the truncated light cone. More precisely, we study inequalities of the form
\[ \|f\|_p\le C_{\alpha,p}(R) \Big\|(\sum_{\gamma\in\Gamma_\alpha(R^{-1})}|f_\gamma|^2)^{1/2}\Big\|_p,  \]
where $\Gamma_\alpha(R^{-1})$ is the set of small caps of width $R^{-\alpha}$. We find sharp upper and lower bounds of the constant $C_{\alpha,p}(R)$.

\end{abstract}

\maketitle


\section{Introduction}
In this paper, we study the square function estimates.
We begin with the most general setting. 
Let $\Omega\subset \R^n$ be a set in the frequency space, and suppose we are given a partition of $\Omega$ into subsets $\Si=\{\si\}$:
$$ \Omega=\bigsqcup_{\si\in\Si}\si. $$

We will only consider the case when $\si$ are morally rectangles.
For any function $f$, we define $f_\si=(\psi_\si \wh f)^\vee$, where $\psi_\si$ is a smooth bump function adapted to $\si$. We will also assume $\supp\wh f\subset \Omega$ in the following discussions. 
The inequality we are interested in is of the following form:

\medskip

\noindent
{\bf Square function estimate}:
$$ \|f\|_p\le C_{p,\Si}\|(\sum_{\si\in\Si}|f_\si|^2)^{1/2}\|_p. $$
The goal is to find the best constant $C_{p,\Si}$ that works for all test functions $f$.

\medskip

This type of estimate is of huge interest in harmonic analysis.
We briefly review some well-known results. 

When $\Omega$ is the $R^{-1}$-neighborhood of the unit parabola $\cP=\{(\xi,\xi^2)\in\R^2: |\xi|\le 1\}$ and $\Si=\{\si\}$ is the set of $\sim R^{-1/2}\times R^{-1}$-caps that form a partition of $\Omega$, then an argument of C\'ordoba-Fefferman (see also \cite[Proposition 3.3]{demeter2020fourier}) gives 
$$ \|f\|_4\lesim\|(\sum_{\si\in\Si}|f_\si|^2)^{1/2}\|_4. $$
(Throughout this article, we suppress the $\sim$ symbol for simplicity when the precise scale is unimportant.)

When $\Omega$ is the $R^{-1}$-neighborhood of the unit cone $\cC=\{(\xi_1,\xi_2,\xi_3)\in\R^3: \xi_3=\sqrt{\xi_1^2+\xi_2^2}, 1/2\le\xi_3\le 1\}$ and $\Si=\{\si\}$ are $1\times R^{-1/2}\times R^{-1}$-caps that form a partition of $\Omega$, then the sharp $L^4$ square function estimate was proved by Guth-Wang-Zhang \cite{guth2020sharp}:
$$ \|f\|_4\lessapprox\|(\sum_{\si\in\Si}|f_\si|^2)^{1/2}\|_4. $$
Here, $A\lessapprox B$ means $A\lesim_\e R^\e B$ for any $\e>0$.

When $\Omega$ is certain neighborhood of a moment curve, it was studied by Gressman, Guo, Pierce, Roos and Yung \cite{gressman2021reversing}. The sharp $L^7$ estimate was obtained by Maldague \cite{maldague2022sharp}. There are some other related results (see \cite{pierce2021superorthogonality}, \cite{gressman2022new}).

\medskip

In the discussion above, we see that the size of caps in the partition of parabola is $R^{-1/2}\times R^{-1}$; the size of caps in the partition of cone is $1\times R^{-1/2}\times R^{-1}$. We usually call them the canonical partition.
Besides the canonical partition of parabola and cone, Demeter, Guth and Wang \cite{demeter2020small} introduced the ``small cap decoupling" which is the decoupling inequality for a finer partition than the canonical partition. Similarly, we can also ask the question about the \textit{small cap square function estimate}.

The goal of this paper is to prove the sharp square function estimates for the small caps of parabola and cone. We will first define the small caps. Then we will introduce and study examples which give sharp lower bounds of the constants. Finally, we will prove the sharp bounds of the constants. 

\subsection{Small caps}
\subsubsection{Small caps for parabola}
Let $\cP:=\{(\xi,\xi^2):\xi\in\R,|\xi|\le 1\}$ be the unit parabola, and $N_{R^{-1}}(\cP)$ be its $R^{-1}$-neighborhood.
For $1/2\le \alpha\le 1$, let $\Ga_\alpha(R^{-1})$ be the partition of $N_{R^{-1}}(\cP)$ into rectangular boxes of dimensions $R^{-\alpha}\times R^{-1}$. More precisely, each $\ga\in\Ga_\alpha(R^{-1})$ is of form
\[ \ga=(I\times \R)\cap N_{R^{-1}}(\cP), \]
where $I\subset [-1,1]$ is an interval of length $R^{-\alpha}$.
Note that we have $\#\Ga_\alpha(R^{-1})\sim R^\alpha$.
Our square function estimate is 
\begin{theorem}\label{thmparabola}
For $\supp\wh f\subset N_{R^{-1}}(\cP)$, we have
\begin{equation}\label{ineqparabola}
    \|f\|_{L^p(\R^2)}\lessapprox
    C_{\alpha,p}(R)\Big\|(\sum\limits_{\ga\in\Ga_\alpha(R^{-1})}|f_\ga|^2)^{1/2}\Big\|_{L^p(\R^2)},
\end{equation}
where
\begin{equation}\label{Cp}
    C_{\alpha,p}(R)= 
    \begin{cases}
     R^{\alpha(\frac{1}{2}-\frac{2}{p})}\ \ \ \ \ p\ge 4\alpha+2,\\
     R^{(\alpha-\frac{1}{2})(\frac{1}{2}-\frac{1}{p})}\ \ \ \ \ 2\le p\le 4\alpha+2.
    \end{cases}
\end{equation}
\end{theorem}
\begin{remark}
{\rm 
We remark that $p\ge 4\alpha+2$ is equivalent to $\alpha(\frac{1}{2}-\frac{2}{p})\ge (\alpha-\frac{1}{2})(\frac{1}{2}-\frac{1}{p})$. Therefore, \eqref{Cp} is equivalent to (up to constant) $C_{\alpha,p}(R)\sim  R^{\alpha(\frac{1}{2}-\frac{2}{p})}+ R^{(\alpha-\frac{1}{2})(\frac{1}{2}-\frac{1}{p})}.$
}
\end{remark}

\subsubsection{Small caps for cone}
Denote the truncated cone in $\R^3$ by $$\cC:=\{(\xi_1,\xi_2,\xi_3)\in\R^3: \xi_3=\sqrt{\xi_1^2+\xi_2^2}, 1/2\le\xi_3\le 1\}.$$
For $1/2\le \be\le 1$, let $\Ga_{\be}(R^{-1})$ be the partition of $N_{R^{-1}}(\cC)$ into caps of dimensions $1\times R^{-\be}\times R^{-1}$. More precisely, we first choose a partition of $\ZS^1$ into $R^{-\be}$-arcs: $\ZS^1=\sqcup\si$. For each arc $\si$, consider the $R^{-1}$-neighborhood of 
$$ \{(\xi_1,\xi_2,\xi_3)\in\cC: \frac{(\xi_1,\xi_2)}{\sqrt{\xi_1^2+\xi_2^2}}\in \si\}, $$
which is a cap of dimensions $1\times R^{-\be}\times R^{-1}$. $\Ga_{\be}(R^{-1})$ is the set of caps constructed in this way (see Figure \ref{conecap}). Note that $\#\Ga_{\be}(R^{-1})\sim R^{\be}$. Our square function estimate is

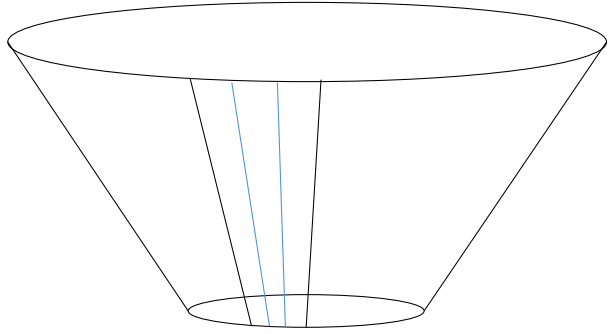
\begin{figure}
    \centering
    
    \begin{tikzpicture}[x=0.75pt,y=0.75pt,yscale=-1,xscale=1]

\draw   (85,65) .. controls (85,53.95) and (152.6,45) .. (236,45) .. controls (319.4,45) and (387,53.95) .. (387,65) .. controls (387,76.05) and (319.4,85) .. (236,85) .. controls (152.6,85) and (85,76.05) .. (85,65) -- cycle ;
\draw   (176,200.75) .. controls (176,196.19) and (202.64,192.5) .. (235.5,192.5) .. controls (268.36,192.5) and (295,196.19) .. (295,200.75) .. controls (295,205.31) and (268.36,209) .. (235.5,209) .. controls (202.64,209) and (176,205.31) .. (176,200.75) -- cycle ;
\draw    (85,65) -- (176,200.75) ;
\draw    (387,65) -- (295,200.75) ;
\draw    (177,83.5) -- (208,208.5) ;
\draw    (243,84) -- (235.5,209) ;
\draw [color={rgb, 255:red, 74; green, 144; blue, 226 }  ,draw opacity=1 ]   (198,85.5) -- (217,208.5) ;
\draw [color={rgb, 255:red, 74; green, 144; blue, 226 }  ,draw opacity=1 ]   (221,85.5) -- (221.94,114.44) -- (225,208.5) ;

\end{tikzpicture}

    \caption{Small caps of the cone}
    \label{conecap}
\end{figure}

\begin{theorem}\label{thmvcone}
For $\supp\wh f\subset N_{R^{-1}}(\cC)$, we have
\begin{equation}\label{ineqvcone}
    \|f\|_{L^p(\R^3)}\lessapprox
    C_{\be,p}(R)\Big\|(\sum\limits_{\ga\in\Ga_{\be}(R^{-1})}|f_\ga|^2)^{1/2}\Big\|_{L^p(\R^3)},
\end{equation}
where
\begin{equation}\label{cbeta}
    C_{\be,p}(R)= 
    \begin{cases}
      R^{\frac{\be}{2}}\ \ \ \ \ p\ge 8,\\
     R^{\frac{\be}{2}+\frac{1}{4}-\frac{2}{p}}\ \ \ \ \ 4\le p\le8\\
     R^{(\be-\frac{1}{2})(1-\frac{2}{p})}\ \ \ \ \ 2\le p\le 4.
    \end{cases}
\end{equation}
\end{theorem}

\begin{remark}\label{rmk}
{\rm 
We remark that there is no interpolation argument in the proof of square function estimate. It is because that we cannot rewrite our square function estimate in the form of
\[ \|Tg\|_X\lesim C \|g\|_Y, \]
where $X, Y$ are some normed vector spaces and $T$ is a linear operator. Another way to see the interpolation argument is prohibited is by looking at the numerology in \eqref{cbeta}. We draw the graph of $(\frac{1}{p},\log_R C_{\beta,p}(R))$, where we ignore the $C_\e R^\e$ factor in $C_{\beta,p}(R)$ (See Figure \ref{graph}). We see
the critical exponent $p=8$ corresponds to a concave point $(\frac{1}{8},\frac{\beta}{2})$ in the graph. But if the interpolation argument works, then the graph should be convex which is a contradiction. Not being allowed to do interpolation will be the main difficulty in the proof. This means that we need to prove the estimate for all $p$, but not only the critical $p$. Let us consider the case $\beta=1/2$. One critical exponent $p=4$ was proved by Guth-Wang-Zhang \cite{guth2020sharp}. The result for another critical exponent $p=8$ and hence for $p\in(4,8)$ is not included in \cite{guth2020sharp}. We also remark that
$$C_{\be,p}(R)\sim  \min\{ R^{\frac{\be}{2}}, R^{\frac{\be}{2}+\frac{1}{4}-\frac{2}{p}}+ R^{(\be-\frac{1}{2})(\frac{1}{2}-\frac{1}{p})} \}.$$
}
\end{remark}

\begin{figure}

\begin{tikzpicture}[x=0.75pt,y=0.75pt,yscale=-1,xscale=1]

\draw  (50.37,236.24) -- (420.63,236.24)(87.39,16) -- (87.39,260.71) (413.63,231.24) -- (420.63,236.24) -- (413.63,241.24) (82.39,23) -- (87.39,16) -- (92.39,23)  ;
\draw    (86.69,91.76) -- (178.5,91.76) ;
\draw    (223.9,176.23) -- (304.61,236.73) ;
\draw    (178.5,91.76) -- (223.9,176.23) ;

\draw (120.99,116.01) node   [align=left] {\begin{minipage}[lt]{68.61pt}\setlength\topsep{0pt}

\end{minipage}};
\draw (444.93,216.54) node [anchor=north west][inner sep=0.75pt]    {$\frac{1}{p}$};
\draw (214.9,240.94) node [anchor=north west][inner sep=0.75pt]    {$\frac{1}{4}$};
\draw (294.61,241.94) node [anchor=north west][inner sep=0.75pt]    {$\frac{1}{2}$};
\draw (-1,20) node [anchor=north west][inner sep=0.75pt]    {$\log_{R}\big( C_{\beta,p}(R)\big)$};

\end{tikzpicture}

\caption{}
    \label{graph}
\end{figure}



\subsection{Elementary tools}

We briefly introduce the notion of \textit{dual rectangle} and \textit{local orthogonality}.

\begin{definition}

Let $R$ be a rectangle of dimensions $a\times b\times c$. Then the dual rectangle of $R$, denoted by $R^*$, is the rectangle centered at the origin of dimensions $a^{-1}\times b^{-1}\times c^{-1}$. Here $R^*$ is made from $R$ by letting the length of each edge of $R$ become the reciprocal.
\end{definition}

From our definition, we see that if $R_2$ is a translated copy of $R_1$, then $R_1^*=R_2^*$. The motivation for defining dual rectangle is the following result.

\begin{lemma}
    For any rectangle $R$, there exists a smooth function $\om_R$ which satisfies $\frac{1}{10}\cdot\Id_R(x)\le \om_R(x)\le 10\cdot\Id_R(x)$ for $x\in R$, and $\om_R$ decays rapidly outside $R$. Also, $\supp\wh \om_R \subset R^*$.
\end{lemma}
This lemma is very standard, so we omit the proof. The next result is the local orthogonality property.

\begin{lemma}
    Let $R$ be a rectangle and $\{f_i\}$ is a set of functions. If $\{ \supp\wh f_i+R^* \}$ are finitely overlapping, then
    \begin{equation}\label{ineqlocal}
        \int_R |\sum f_i|^2\lesim \int \sum|f_i|^2 |\om_R|^2.
    \end{equation} 
\end{lemma}
\begin{proof}
    \begin{align*}
        \int_R |\sum f_i|^2\lesim \int |\sum f_i \om_R|^2=\int |\sum \wh{f_i\om_R}|^2.
    \end{align*}
Note that $\wh{f_i\om_R}=\wh{f_i}*\wh{\om_R}$ is supported in $\supp \wh f_i+R^*$. By the finitely overlapping property, we see the above is bounded by
\[ \lesim \int \sum|\wh{f_i\om_R}|^2=\int \sum |f_i \om_R|^2. \]
\end{proof}

\begin{remark}
{\rm
    Note that $\om_R$ is essentially $\Id_R$ by ignoring the rapidly decaying tail. It turns out that the tail is always harmless. Therefore, to get rid of some irrelevant technicalities, we will just ignore the rapidly decaying tail, and write \eqref{ineqlocal} as
\[ \int_R |\sum f_i|^2\lesim \int_R \sum|f_i|^2. \]

}
\end{remark}

There is another notion called \textit{comparable}. Given two rectangles $R_1, R_2$, we say $R_1$ is essentially contained in $R_2$, if there exists a universal constant $C$ (say $C=100$) such that 
\[ R_1\subset C R_2. \]
We say $R_1$ and $R_2$ are comparable if $R_1$ is essentially contained in $R_2$ and vice versa, i.e.,
\[ \frac{1}{C} R_1\subset R_2\subset C R_1. \]
Throughout this paper, we will just ignore the unimportant constant $C$, and just write $R_1\subset R_2$ to denote that $R_1$ is essentially contained in $R_2$.

\bigskip

\begin{sloppypar}
\noindent {\bf Acknowledgement.}
The author would like to thank Prof. Larry Guth and Dominique Maldague for helpful discussions. The author also want to thank Changkeun Oh for the discussion of Theorem 1. The author also want to thank referees for carefully reading the manuscript and many helpful suggestions.
\end{sloppypar}

\section{Small cap square function estimate for parabola}
We prove Theorem \ref{thmparabola} in this section. We begin with the sharp examples.

\subsection{Sharp examples} There are two types of examples: \textit{concentrated example} and \textit{flat example}. 

\medskip

\noindent\fbox{Case 1: $p\ge 4\alpha+2$}

We introduce the concentrated example. Choose $f$ such that $\wh f(\xi)=\psi_{N_{R^{-1}}(\cP)}(\xi)$, where $\psi_{N_{R^{-1}}(\cP)}$ is a smooth bump function supported in $N_{R^{-1}}(\cP)$. We see that $f(0)=\int\wh f(\xi)d\xi \sim R^{-1}$. Since $\wh f$ is supported in the unit ball centered at the origin, $f$ is locally constant in $B(0,1)$. Therefore,
$$ \|f\|_p\ge \|f\|_{L^p(B(0,1))}\gtrsim R^{-1}. $$
We consider the right hand side of \eqref{ineqparabola}. By definition, for each $\ga\in\Ga_\alpha(R^{-1})$, $\wh f_\ga$ is roughly a bump function supported in $2\ga$. Let $\ga^*$ be the dual rectangle of $\ga$ which has dimensions $R^\alpha\times R$ and is centered at the origin. By an application of integration by parts and by ignoring the tails, we assume 
$$ f_\ga= \frac{1}{|\ga^*|}\Id_{\ga^*}. $$
Here, ``$\approx$" means up to a $C_\e R^\e$ factor for any $\e>0$. We will use the same notation throughout the paper.

We see that
$$ \Big\|(\sum\limits_{\ga\in\Ga_\alpha(R^{-1})}|f_\ga|^2)^{1/2}\Big\|_{L^p(\R^2)}^p\sim R^{-(1+\alpha)p}\int_{B(0,R)} (\sum_\ga \Id_{\ga^*})^{p/2}. $$
We evaluate the integral above. There are two extreme regions: $B(0,R^\alpha)$ where all the $\{\ga^*\}$ overlap; $B(0,R)\setminus B(0,R/2)$ where $\{\ga^*\}$ is $ O(R^{2\alpha-1})$-overlapping. For the intermediate region $B(0,r)\setminus B(0,r/2)$ ($R^\alpha\le r\le R$), we see that $\{\ga^*\}$ is $ O(r^{-1}R^{2\alpha})$-overlapping. We may find a dyadic radius $r$ such that
\begin{align*}
    \int (\sum_\ga \Id_{\ga^*})^{p/2}&\approx  \int_{B(0,r)\setminus B(0,r/2)} (\sum_\ga \Id_{\ga^*})^{p/2}\lesssim (r^{-1}R^{2\alpha})^{p/2}|B(0,r)|\sim
    r^{2-\frac{p}{2}}R^{\alpha p}.
\end{align*} 
Since $p\ge 4\alpha+2\ge 4$, the expression above is maximized when $r=R^\alpha$. Plugging in, we obtain
$$ \int (\sum_\ga \Id_{\ga^*})^{p/2}\lessapprox R^{\alpha(2+\frac{p}{2})}.$$
Plugging into \eqref{ineqparabola}, we have
$$ R^{-1}\lessapprox C_{\alpha,p}(R)R^{-(1+\alpha)} R^{\alpha(\frac{2}{p}+\frac{1}{2})}, $$
which gives
$$ C_{\alpha,p}(R)\gtrapprox  R^{\alpha(\frac{1}{2}-\frac{2}{p})}.$$

\medskip

\noindent\fbox{Case 2: $2\le p\le 4\alpha+2$}

We introduce the flat example. Let $\theta\subset N_{R^{-1}}(\cP)$ be a $R^{-1/2}\times R^{-1}$-cap. Choose $f$ such that $\wh f(\xi)=\psi_\theta(\xi)$, where $\psi_\theta$ is a smooth bump function supported in $N_{R^{-1}}(\cP)$. Let $\theta^*$ be the dual rectangle of $\theta$ which has dimensions $R^{1/2}\times R$ and is centered at the origin. By the locally constant property, $f$ is an $L^1$ normalized function essentially supported in $\theta^*$. By ignoring the tails, we assume
$$ f= \frac{1}{|\theta^*|}\Id_{\theta^*}. $$
We see that
$$ \|f\|_p\sim R^{-\frac{3}{2}}R^{\frac{3}{2p}}. $$

We consider the right hand side of \eqref{ineqparabola}. By the same reasoning as in Case 1, for each $\ga\in\Ga_\alpha(R^{-1})$ with $\ga\subset \theta$, we know that $\wh f_\ga$ is roughly a bump function supported in $2\ga$.
Therefore, we can assume
$$ f_\ga= \frac{1}{|\ga^*|}\Id_{\ga^*}. $$
We also note that $\ga_1^*$ and $\ga_2^*$ are comparable when $\ga_1,\ga_2\subset \theta$.
We have
\begin{align*}
    \Big\|(\sum\limits_{\ga\in\Ga_\alpha(R^{-1})}|f_\ga|^2)^{1/2}\Big\|_{L^p(\R^2)}&\sim R^{-(1+\alpha)}\bigg(\int (\sum_{\ga\subset \theta} \Id_{\ga^*})^{p/2}\bigg)^{1/p}\sim R^{-(1+\alpha)}\#\{\ga\subset\theta\}^{1/2}|\ga^*|^{1/p} \\
    &\sim R^{-(1+\alpha)} R^{\frac{1}{2}(\alpha-\frac{1}{2})}R^{\frac{1+\alpha}{p}}.
\end{align*} 

Plugging into \eqref{ineqparabola}, we have
$$ R^{-\frac{3}{2}}R^{\frac{3}{2p}}\lessapprox C_{\alpha,p}(R)R^{-(1+\alpha)} R^{\frac{1}{2}(\alpha-\frac{1}{2})}R^{\frac{1+\alpha}{p}}, $$
which gives
$$ C_{\alpha,p}(R)\gtrapprox  R^{(\alpha-\frac{1}{2})(\frac{1}{2}-\frac{1}{p})}.$$

\subsection{Proof of Theorem \ref{thmparabola}}
By the standard localization argument, it suffices to prove
$$ \|f\|_{L^p(B_R)}\lessapprox_\e
    (R^{\alpha(\frac{1}{2}-\frac{2}{p})}+R^{(\alpha-\frac{1}{2})(\frac{1}{2}-\frac{1}{p})})\Big\|(\sum\limits_{\ga\in\Ga_\alpha(R^{-1})}|f_\ga|^2)^{1/2}\Big\|_p. $$
We introduce some notations. Throughout the proof, we use $\ga$ to denote caps of dimensions $R^{-\alpha}\times R^{-1}$. For $R^{-1/2}\le \De\le 1$, we will consider caps $\tau$ of length $\De$ and thickness $R^{-1}$. We write $|\tau|=\De$ to indicate the length of $\tau$. We will also partition the region $B_R$ into rectangles of dimensions $R^\alpha\times R$. For simplicity, we denote these rectangles by $B_{R^\alpha\times R}.$ The longest direction of $B_{R^\alpha\times R}$ will be specified in the proof.

Let $K\sim \log R$ and let $m\in\N$ be such that $K^m=R^{1/2}$. By doing the broad-narrow reduction as in \cite[Section 5.1]{demeter2020small} , we have
\begin{align}
    \label{par1}\|f\|_{L^p(B_R)}^p&\lesssim C^m\sum_{|\theta|=R^{-1/2}}\|f_\theta\|_{L^p(B_R)}^p\\
    \label{par2}&+C^m K^C \sum_{\begin{subarray}{c}
        R^{-1/2}\le \De\le 1\\
        \De \text{~dyadic}
    \end{subarray}} \sum_{|\tau|\sim \De} \max_{\begin{subarray}{c}
    \tau_1,\tau_2\subset \tau\\
    |\tau_1|=|\tau_2|=K^{-1}\De\\
    \dist(\tau_1,\tau_2)\ge (10K)^{-1}\De
    \end{subarray}
    }\|(f_{\tau_1}f_{\tau_2})^{1/2}\|_{L^p(B_R)}^p.
\end{align} 
Note that $C^mK^C\lesssim_\e R^\e$, for each $\e>0$.

We first estimate the right hand side of \eqref{par1}.
\begin{lemma}\label{lemnarrow}
Let $\theta$ be a cap of length $R^{-1/2}$. Then,
$$ \|f_\theta\|_{L^p(B_R)}\lesssim R^{(\alpha-\frac{1}{2})(\frac{1}{2}-\frac{1}{p})}\|(\sum_{\ga\subset \theta}|f_\ga|^2)^{1/2}\|_p. $$
\end{lemma}

\begin{proof}
We partition $B_R$ into $B_{R^\alpha\times R}$, where each $B_{R^\alpha\times R}$ is a translation of $\ga^*$ for $\ga\subset \theta$ (note that for all $\ga\subset\theta$, $\ga^*$'s are comparable). It suffices to prove for any $B_{R^\alpha\times R}$,
\begin{equation}\label{lem1suffice}
    \|f_\theta\|_{L^p(B_{R^\alpha\times R})}\lesssim R^{(\alpha-\frac{1}{2})(\frac{1}{2}-\frac{1}{p})}\|(\sum_{\ga\subset \theta}|f_\ga|^2)^{1/2}\|_{L^p(\om_{B_{R^\alpha\times R}})}.
\end{equation}
Here, $\om_{B_{R^\alpha\times R}}$ is a weight which $=1$ on $B_{R^\alpha\times R}$ and decays rapidly outside $B_{R^\alpha\times R}$. And $\|g\|_{L^p(\om)}$ is defined to be $(\int |g|^p\om)^{1/p}$. We remark that we use $\om_{B_{R^\alpha\times R}}$ instead of $\Id_{B_{R^\alpha\times R}}$ is to make the local orthogonality and locally constant property rigorous. As such technicality is well-known (see for example in \cite{demeter2020fourier}), we will just pretend $\om_{B_{R^\alpha\times R}}=\Id_{B_{R^\alpha\times R}}$ for convenience.

We further do the partition
$$ B_{R^\alpha\times R}=\bigsqcup B_{R^{1/2}\times R}, $$
where each $B_{R^{1/2}\times R}$ is a translation of $\theta^*$. Since $f_\theta$ is locally constant on each $B_{R^{1/2}\times R}$, we have
\begin{align*}
    \|f_\theta\|_{L^p(B_{R^\alpha\times R})}&=\bigg( \sum_{B_{R^{1/2}\times R}}\|f_\theta\|_{L^p(B_{R^{1/2}\times R})}^p \bigg)^{1/p}\\
    &\lesssim R^{\frac{3}{2}(\frac{1}{p}-\frac{1}{2})}\bigg( \sum_{B_{R^{1/2}\times R}}\|f_\theta\|_{L^2(B_{R^{1/2}\times R})}^p \bigg)^{1/p}\\\
    &\le R^{\frac{3}{2}(\frac{1}{p}-\frac{1}{2})}\|f_\theta\|_{L^2(B_{R^\alpha\times R})}.
\end{align*} 
By local orthogonality, H\"older's inequality and noting $p\ge 2$, we have
$$ \|f_\theta\|_{L^2(B_{R^\alpha\times R})}\lesssim \|(\sum_{\ga\subset\theta}|f_\ga|^2)^{1/2}\|_{L^2(B_{R^\alpha\times R})}\le R^{(1+\alpha)(\frac{1}{2}-\frac{1}{p})} \|(\sum_{\ga\subset\theta}|f_\ga|^2)^{1/2}\|_{L^p(B_{R^\alpha\times R})}.$$
Combining the inequalities, we finish the proof of \eqref{lem1suffice}.
\end{proof}
By Lemma \ref{lemnarrow}, the right hand side of \eqref{par1} is bounded by
$$ R^\e R^{(\alpha-\frac{1}{2})(\frac{1}{2}-\frac{1}{p})}\bigg(\sum_\theta\|(\sum_{\ga\subset \theta}|f_\ga|^2)^{1/2}\|_p^p\bigg)^{1/p}\le  C_{\alpha,p}(R)\|(\sum_{\ga}|f_\ga|^2)^{1/2}\|_p.$$

\medskip

Next, we estimate \eqref{par2}. For any summand in \eqref{par2}, we will show that
\begin{equation}\label{toshow}
    \|(f_{\tau_1}f_{\tau_2})^{1/2}\|_{L^p(B_R)}\lesim C_{\alpha,p}(R) \|(\sum_{\ga\subset \tau}|f_\ga|^2)^{1/2}\|_p.
\end{equation} 
This will imply \eqref{par2}$^{\frac{1}{p}}$ $\lessapprox C_{\alpha,p}(R)\|(\sum\limits_{\ga
}|f_\ga|^2)^{1/2}\|_p$, and then finishes the proof of Theorem \ref{thmparabola}. It remains to prove \eqref{toshow}.

Fix a $\De\in [R^{-1/2},1]$ and a $\tau$ with $|\tau|=\De$. We first consider $\bigcap_{\ga\subset \tau} \ga^*$. It is easy to see $\bigcap_{\ga\subset \tau} \ga^*$ is an $R^\alpha\times R^\alpha \De^{-1}$-rectangle when $\De\ge R^{\alpha-1}$; $\bigcap_{\ga\subset \tau} \ga^*$ is an $R^\alpha\times R$-rectangle when $\De\le R^{\alpha-1}$. We consider these two cases separately.

\medskip

\noindent\fbox{Case 1: $\De\ge R^{\alpha-1}$}

We choose a partition 
 $B_R=\bigsqcup B_{R^{\alpha}\times R^\alpha\De^{-1}}$, 
where each $B_{R^{\alpha}\times R^\alpha\De^{-1}}$ is a translation of $\bigcap_{\ga\subset\tau}\ga^*$. We just need to show
\begin{equation}\label{toshow1}
    \|(f_{\tau_1}f_{\tau_2})^{1/2}\|_{L^p(B_{R^{\alpha}\times R^\alpha\De^{-1}})}\lesim C_{\alpha,p}(R) \|(\sum_{\ga\subset \tau}|f_\ga|^2)^{1/2}\|_{L^p(B_{R^{\alpha}\times R^\alpha\De^{-1}})}.
\end{equation} 
Since each $|f_\ga|$ is locally constant on $B_{R^{\alpha}\times R^\alpha\De^{-1}}$ when $\ga\subset \tau$, we have
$$ \|(\sum_{\ga\subset \tau}|f_\ga|^2)^{1/2}\|_{L^p(B_{R^{\alpha}\times R^\alpha\De^{-1}})}\sim (R^{2\alpha}\De^{-1})^{-\frac{1}{2}+\frac{1}{p}}\|(\sum_{\ga\subset \tau}|f_\ga|^2)^{1/2}\|_{L^2(B_{R^{\alpha}\times R^\alpha\De^{-1}})}. $$
Since $\{f_\ga\}_{\ga\subset\tau}$ are locally orthogonal on $B_{R^{\alpha}\times R^\alpha\De^{-1}}$, we have
$$ \|(\sum_{\ga\subset \tau}|f_\ga|^2)^{1/2}\|_{L^2(B_{R^{\alpha}\times R^\alpha\De^{-1}})}\sim \|f_\tau\|_{L^2(B_{R^{\alpha}\times R^\alpha\De^{-1}})}. $$
Therefore, \eqref{toshow1} is reduced to 
\begin{equation}\label{toshow2}
    \|(f_{\tau_1}f_{\tau_2})^{1/2}\|_{L^p(B_{R^{\alpha}\times R^\alpha\De^{-1}})}\lesim C_{\alpha,p}(R) (R^{2\alpha}\De^{-1})^{-\frac{1}{2}+\frac{1}{p}} \|f_\tau\|_{L^2(B_{R^{\alpha}\times R^\alpha\De^{-1}})}.
\end{equation} 

Next, we apply the parabolic rescaling. Recall that $\tau$ is a cap of length $\De$. We dilate by factor $\De^{-1}$ in the tangent direction of $\tau$ and dilate by factor $\De^{-2}$ in the normal direction of $\tau$. Under the rescaling, we see that: $\tau$ becomes the $R^{-1}\De^{-2}$-neighborhood of $\cP$; $\tau_1$ and $\tau_2$ become $K^{-1}$-separated caps with length $K^{-1}$ and thickness $R^{-1}\De^{-2}$;
the rectangle $B_{R^{\alpha}\times R^\alpha\De^{-1}}$ in the physical space becomes $B_{R^\alpha \De}$.
Let $g,g_1,g_2$ be the rescaled version of $f_\tau,f_{\tau_1},f_{\tau_2}$ respectively.
The inequality \eqref{toshow2} becomes
\begin{equation}\label{toshow3}
    \|(g_1 g_2)^{1/2}\|_{L^p(B_{R^{\alpha}\De})}\lesim C_{\alpha,p}(R) (R^{2\alpha}\De^{-1})^{-\frac{1}{2}+\frac{1}{p}} \De^{3(-\frac{1}{2}+\frac{1}{p})}\|g\|_{L^2(B_{R^{\alpha}\De})}.
\end{equation}
We recall the following bilinear restriction estimate (see for example in \cite{tao2003sharp}).
\begin{lemma}\label{birestriction}
Let $r>1$, $K>1$. Suppose $g_1, g_2$ satisfy $\supp \wh g_1, \supp\wh g_2\subset N_{r^{-2}}(\cP)$ and $\dist(\supp \wh g_1, \supp \wh g_2)>K^{-1}$. Then for $p\ge 2$ and $r'\ge r$ we have
\begin{equation}
   \|(g_1g_2)^{1/2}\|_{L^p(B_{r'})}\lesssim K^{O(1)}r^{\frac{2}{p}-1}\big(\|g_1\|_{L^2(B_{r'})}\|g_2\|_{L^2(B_{r'})}\big)^{1/2}.  
\end{equation}
\end{lemma}
\begin{proof}
We just need to prove for $r'=r$.
When $p=2$, this is trivial. When $p=4$, this is the bilinear restriction estimate. When $p=\infty$, we note that 
\begin{align*}
    \|(g_1g_2)^{1/2}\|_{L^\infty(B_r)}^2&\le \|g_1\|_{L^\infty(B_r)}\|g_2\|_{L^\infty(B_r)}\le \|\wh g_1\|_{L^1}\|\wh g_2\|_{L^1}\\
    &\lesim r^{-2}\|\wh g_1\|_{L^2}\|\wh g_1\|_{L^2}=r^{-2}\|g_1\|_{L^2}\|g_2\|_{L^2}.
\end{align*} 
The second-last inequality is by H\"older and the condition on the support of $\wh g_1,\wh g_2$. The last inequality is by Plancherel.
For other $p$, the proof is by using H\"older to interpolate between $p=2,4,\infty$.
\end{proof}

We return to \eqref{toshow3}. Noting that $R^\alpha \De\ge (R\De^2)^{1/2}$, we apply the lemma above to bound the left hand side of \eqref{toshow3} by $(R^\alpha \De)^{\frac{2}{p}-1}\|g\|_{L^2(B_{R^\alpha \De})}$. It suffices to prove
\begin{equation}\label{toshow4}
    (R^\alpha \De)^{\frac{2}{p}-1}\lesim C_{\alpha,p}(R) (R^{2\alpha}\De^{-1})^{-\frac{1}{2}+\frac{1}{p}} \De^{3(-\frac{1}{2}+\frac{1}{p})}.
\end{equation} 

When $p\ge 4$, we use $C_{\alpha,p}(R)\gtrsim R^{\alpha(\frac{1}{2}-\frac{2}{p})}$. Then \eqref{toshow4} boils down to
\begin{equation}
    (R^\alpha \De)^{\frac{2}{p}-1}\lesim R^{\alpha(\frac{1}{2}-\frac{2}{p})} (R^{2\alpha}\De^{-1})^{-\frac{1}{2}+\frac{1}{p}} \De^{3(-\frac{1}{2}+\frac{1}{p})},
\end{equation} 
which is equivalent to
$$ R^{\alpha(\frac{1}{2}-\frac{2}{p})}\gtrsim 1, $$
which is true since $R\ge 1$.

When $p\le 4$, we use $C_{\alpha,p}(R)\gtrsim R^{(\alpha-\frac{1}{2})(\frac{1}{2}-\frac{1}{p})}$.
Then \eqref{toshow4} boils down to
\begin{equation}
    (R^\alpha \De)^{\frac{2}{p}-1}\lesim R^{(\alpha-\frac{1}{2})(\frac{1}{2}-\frac{1}{p})} (R^{2\alpha}\De^{-1})^{-\frac{1}{2}+\frac{1}{p}} \De^{3(-\frac{1}{2}+\frac{1}{p})},
\end{equation} 
which is equivalent to
$$ R^{(\alpha-\frac{1}{2})(\frac{1}{2}-\frac{1}{p})}\gtrsim 1, $$
which is true since $\alpha\ge 1/2.$

\medskip

\noindent\fbox{Case 2: $\De\le R^{\alpha-1}$}

We choose a partition 
 $B_R=\bigsqcup B_{R^{\alpha}\times R}$, 
where each $B_{R^{\alpha}\times R}$ is a translation of $\bigcap_{\ga\subset\tau}\ga^*$. We just need to show
\begin{equation}\label{toshow11}
    \|(f_{\tau_1}f_{\tau_2})^{1/2}\|_{L^p(B_{R^{\alpha}\times R})}\lesim C_{\alpha,p}(R) \|(\sum_{\ga\subset \tau}|f_\ga|^2)^{1/2}\|_{L^p(B_{R^{\alpha}\times R})}.
\end{equation} 
Since each $|f_\ga|$ is locally constant on $B_{R^{\alpha}\times R}$ when $\ga\subset \tau$, we have
$$ \|(\sum_{\ga\subset \tau}|f_\ga|^2)^{1/2}\|_{L^p(B_{R^{\alpha}\times R})}\sim (R^{\alpha+1})^{-\frac{1}{2}+\frac{1}{p}}\|(\sum_{\ga\subset \tau}|f_\ga|^2)^{1/2}\|_{L^2(B_{R^{\alpha}\times R})}. $$
Since $\{f_\ga\}_{\ga\subset\tau}$ are locally orthogonal on $B_{R^{\alpha}\times R}$, we have
$$ \|(\sum_{\ga\subset \tau}|f_\ga|^2)^{1/2}\|_{L^2(B_{R^{\alpha}\times R})}\sim \|f_\tau\|_{L^2(B_{R^{\alpha}\times R})}. $$
Therefore, \eqref{toshow11} is reduced to 
\begin{equation}\label{toshow12}
    \|(f_{\tau_1}f_{\tau_2})^{1/2}\|_{L^p(B_{R^{\alpha}\times R})}\lesim C_{\alpha,p}(R) (R^{\alpha+1})^{-\frac{1}{2}+\frac{1}{p}} \|f_\tau\|_{L^2(B_{R^{\alpha}\times R})}.
\end{equation} 

Next, we do the same parabolic rescaling as above. The rectangle $B_{R^{\alpha}\times R}$ in the physical space becomes $B_{R^\alpha \De\times R\De^2}$.
Let $g,g_1,g_2$ be the rescaled version of $f_\tau,f_{\tau_1},f_{\tau_2}$ respectively.
The inequality \eqref{toshow12} becomes
\begin{equation}\label{toshow13}
    \|(g_1 g_2)^{1/2}\|_{L^p(B_{R^{\alpha}\De\times R\De^2})}\lesim C_{\alpha,p}(R) (R^{\alpha+1})^{-\frac{1}{2}+\frac{1}{p}} \De^{3(-\frac{1}{2}+\frac{1}{p})}\|g\|_{L^2(B_{R^{\alpha}\De\times R\De^2})}.
\end{equation}

To apply Lemma \ref{birestriction}, we do the partition $B_{R^\alpha\De\times R\De^2}=\bigsqcup B_{R\De^2}$. So, \eqref{toshow13} is reduced to 
\begin{equation}\label{toshow13.5}
    \|(g_1 g_2)^{1/2}\|_{L^p(B_{R\De^2})}\lesim C_{\alpha,p}(R) (R^{\alpha+1})^{-\frac{1}{2}+\frac{1}{p}} \De^{3(-\frac{1}{2}+\frac{1}{p})}\|g\|_{L^2(B_{R\De^2})}.
\end{equation}
By Lemma \ref{birestriction},  
$$\|(g_1 g_2)^{1/2}\|_{L^p(B_{R\De^2})}\lesim (R \De^2)^{\frac{2}{p}-1}\|g\|_{L^2(B_{R\De^2})}.$$ It suffices to prove
\begin{equation}\label{toshow14}
    (R \De^2)^{\frac{2}{p}-1}\lesim C_{\alpha,p}(R) R^{(\alpha+1)(-\frac{1}{2}+\frac{1}{p})} \De^{3(-\frac{1}{2}+\frac{1}{p})}.
\end{equation} 

When $p\ge 4\alpha+2$, we use $C_{\alpha,p}(R)\gtrsim R^{\alpha(\frac{1}{2}-\frac{2}{p})}$. Then \eqref{toshow14} boils down to
\begin{equation}
    (R\De^2)^{\frac{2}{p}-1}\lesim R^{\alpha(\frac{1}{2}-\frac{2}{p})} R^{(\alpha+1)(-\frac{1}{2}+\frac{1}{p})} \De^{3(-\frac{1}{2}+\frac{1}{p})},
\end{equation} 
which is equivalent to
$$ \De^{\frac{1}{p}-\frac{1}{2}}\lesim R^{-\frac{\alpha}{p}+\frac{1}{2}-\frac{1}{p}}. $$
Using $\De\ge R^{-\frac{1}{2}}$, we just need to prove
$$ R^{-\frac{1}{2p}+\frac{1}{4}}\lesim R^{-\frac{\alpha}{p}+\frac{1}{2}-\frac{1}{p}}. $$
The last inequality is equivalent to $\frac14-\frac{1}{2p}-\frac\alpha p\ge 0$, which is further equivalent to $p\ge 4\alpha+2$. We also remark that this is the place where the critical exponent $p=4\alpha+2$ appears.

When $2\le p\le 4\alpha+2$, we use $C_{\alpha,p}(R)\gtrsim R^{(\alpha-\frac{1}{2})(\frac{1}{2}-\frac{1}{p})}$.
Then \eqref{toshow14} boils down to
\begin{equation}
    (R\De^2)^{\frac{2}{p}-1}\lesim R^{(\alpha-\frac{1}{2})(\frac{1}{2}-\frac{1}{p})} R^{(\alpha+1)(-\frac{1}{2}+\frac{1}{p})} \De^{3(-\frac{1}{2}+\frac{1}{p})},
\end{equation} 
which is equivalent to
$$ \De^{\frac{1}{p}-\frac{1}{2}}\lesim R^{\frac{1}{4}-\frac{1}{2p}}, $$
which is true since $\De^{-1}\le R^{1/2}.$

The proof of Theorem \ref{thmparabola} is finished.

\newpage





\section{Small cap square function estimate for cone}

We prove Theorem \ref{thmvcone} in this section. We begin with the sharp examples.

\subsection{Sharp examples}\hfill

\smallskip

Choose $f$ such that $\wh f=\psi_{N_{R^{-1}}(\cC)}(\xi)$, where $\psi_{N_{R^{-1}}(\cC)}(\xi)$ is a smooth bump function supported in $N_{R^{-1}}(\cC)$. We are going to calculate the lower bound of $\|f\|_{p}$, which is the left hand side of \eqref{ineqvcone}. We see that $f(0)=\int \wh f(\xi) d\xi\sim R^{-1}$. Since $\wh f$ is supported in the unit ball centered at the origin, $f$ is locally constant in $B(0,1)$. Therefore,
\begin{equation}\label{combine1}
    \|f\|_p\gtrsim \|f\|_{L^p(B(0,1))}\gtrsim R^{-1}. 
\end{equation} 
We also estimate the integral of $f$ in the region $\{|x|\sim R\}$. We first do a canonical partition of $N_{R^{-1}}(\cC)$ into $1\times R^{-1/2}\times R^{-1}$-planks, denoted by
\[ N_{R^{-1}}(\cC)=\bigsqcup \theta. \]
Then we can write $f=\sum_\theta f_\theta$, such that each $\wh f_\theta$ is a smooth bump function on $\theta$. Let $\theta^*$ be the dual rectangle of $\theta$, so $\theta^*$ has size $1\times R^{1/2}\times R$ and is centered at the origin. By an application of integration by parts, we can assume
\[ |f_\theta|= \frac{1}{|\theta^*|}\Id_{\theta^*}=R^{-3/2}\Id_{\theta^*}. \]
Now the key observation is that $\{\theta^*\}$ are disjoint in $B(0,R)\setminus B(0,\frac{9}{10}R)$, so we see that
\begin{align*}
    \|f\|_p=\|\sum_\theta f_\theta\|_p&\ge \|\sum_\theta f_\theta\|_{L^p(B(0,R)\setminus B(0,\frac{9}{10}R))}\\
    &\sim R^{-3/2}\| \sum_\theta \Id_{\theta^*} \|_{L^p(B(0,R)\setminus B(0,\frac{9}{10}R))}\\
    &\sim R^{-3/2}(\sum_\theta |\theta^*| )^{1/p}=R^{-\frac{3}{2}+\frac{2}{p}}.
\end{align*} 
Combining with \eqref{combine1}, we see
\begin{equation}\label{lowbound}
    \|f\|_p\gtrsim \max\{ R^{-1}, R^{-\frac{3}{2}+\frac{2}{p}} \}.
\end{equation}
And we see the threshold for these two lower bounds to be equal is at $p=4$.

\medskip

For this same $f$, we will estimate the upper bound of the right hand side of \eqref{ineqvcone}. Recall that $\ga$ is a $1\times R^{-\beta}\times R^{-1}$-cap contained in $N_{R^{-1}}(\cC)$, and by definition $\wh f_\ga=\psi_\ga \wh f$. Therefore, $\wh f_\ga$ is a smooth bump function adapted to $\ga$. By an application of integration by parts, we can assume
\[ |f_\ga|= \frac{1}{|\ga^*|}\Id_{\ga^*}. \]
Here, the dual rectangle $\ga^*$ is centered at the origin with size $1\times R^\beta\times R$.
See Figure \ref{fig3}: the rectangle on the left hand side is $\ga$; the rectangle on the right hand side is $\ga^*$.

Therefore, we can write
\begin{equation}\label{upbound}
    \|(\sum_{\ga\in\Ga_\beta(R^{-1})}|f_\ga|^2)^{1/2}\|_p\sim R^{-1-\beta} \big(\int(\sum_\ga \Id_{\ga^*})^{p/2}\big)^{1/p}. 
\end{equation} 

\newpage

\begin{figure}
    \centering
    
\begin{tikzpicture}[x=0.75pt,y=0.75pt,yscale=-1,xscale=1]

\draw   (46,107) .. controls (46,95.95) and (91.67,87) .. (148,87) .. controls (204.33,87) and (250,95.95) .. (250,107) .. controls (250,118.05) and (204.33,127) .. (148,127) .. controls (91.67,127) and (46,118.05) .. (46,107) -- cycle ;
\draw   (117,203.25) .. controls (117,198.97) and (132.67,195.5) .. (152,195.5) .. controls (171.33,195.5) and (187,198.97) .. (187,203.25) .. controls (187,207.53) and (171.33,211) .. (152,211) .. controls (132.67,211) and (117,207.53) .. (117,203.25) -- cycle ;
\draw    (46,107) -- (117,203.25) ;
\draw    (250,107) -- (187,203.25) ;
\draw   (163.16,89.1) -- (169.26,84.16) -- (189.92,86.32) -- (178.75,192.73) -- (172.65,197.67) -- (152,195.5) -- cycle ; \draw   (189.92,86.32) -- (183.82,91.26) -- (163.16,89.1) ; \draw   (183.82,91.26) -- (172.65,197.67) ;
\draw   (516.64,207.13) -- (464.93,202.89) -- (369.85,90.82) -- (382.04,80.48) -- (433.76,84.73) -- (528.83,196.79) -- cycle ; \draw   (369.85,90.82) -- (421.57,95.07) -- (516.64,207.13) ; \draw   (421.57,95.07) -- (433.76,84.73) ;

\end{tikzpicture}
    
    \caption{}
    \label{fig3}
\end{figure}

Note that each $\ga^*$ is supported in $B(0,R)$, so we rewrite 
\[\int(\sum_\ga \Id_{\ga^*})^{p/2}=\int_{|r|\le R}dr\int_{\{x_3=r\}}(\sum_\ga \Id_{\ga^*})^{p/2}dx_1dx_2.\]
We are going to calculate $\int_{\{x_3=r\}}(\sum_\ga \Id_{\ga^*})^{p/2}$. Here is the result:

\begin{proposition} For $p\ge 2$, we have
\begin{equation}\label{slice}
    \int_{\{x_3=r\}}(\sum_\ga \Id_{\ga^*})^{p/2}\approx\begin{cases}
        R^{2\beta}+R^{\frac{p\beta}{2}} & 0\le r\le 10,\\
        r^{1-\frac{p}{4}}R^{\beta\frac{p}{2}}+r^{2-\frac{p}{2}}R^{\beta\frac{p}{2}}+R^{2\beta}  & 10\le r\le R^\beta,\\
         r^{1-\frac{p}{4}}R^{\beta\frac{p}{2}}+R^{2\beta} & R^\beta\le r\le R.
    \end{cases}
\end{equation}
\end{proposition}

\begin{figure}
    \centering
    \begin{tikzpicture}[x=0.75pt,y=0.75pt,yscale=-1,xscale=1]

\draw   (159,153) .. controls (159,106.61) and (196.61,69) .. (243,69) .. controls (289.39,69) and (327,106.61) .. (327,153) .. controls (327,199.39) and (289.39,237) .. (243,237) .. controls (196.61,237) and (159,199.39) .. (159,153) -- cycle ;
\draw   (153.5,111) -- (164.5,111) -- (164.5,195) -- (153.5,195) -- cycle ;
\draw   (167.35,94.12) -- (177.92,97.17) -- (154.65,177.88) -- (144.08,174.83) -- cycle ;
\draw   (180.99,76.02) -- (190.81,80.97) -- (153.01,155.98) -- (143.19,151.03) -- cycle ;
\end{tikzpicture}

    \caption{}
    \label{fig4}
\end{figure}

\begin{proof}
Fix the plane $\{x_3=r\}$.
For each $\ga^*$, we set 
\[ \ga^*_r:=\ga^*\cap \{x_3=r\}. \]
$\ga^*_r$ is a rectangle of size $1\times R^\beta$ in the plane $\{x_3=r\}$. Denote the center of $\ga^*_r$ by $C(\ga^*_r)$. We see that $C(\ga^*_r)$ lies on the circle 
\[S_r:=\{x_3=r, \sqrt{x_1^2+x_2^2}=r\},\]
and the long direction of $\ga^*_r$ is tangent to $S_r$ (see Figure \ref{fig4}).
We can rewrite the left hand side of \eqref{slice} as
\[ \int_{\R^2} (\sum_\ga \Id_{\ga^*_r})^{p/2}. \]
We also notice two useful facts: (1) $\#\{\ga^*_r\}\sim R^\beta$; (2) $\{C(\ga^*_r)\}$ are roughly $rR^{-\beta}$-separated on the circle $S_r$.

\medskip

\noindent\fbox{Case 1: $0\le r\le 10$}  

In this case, we see that $\{\ga^*_r\}$ essentially form a bush centered at the origin. Evaluating the concentrated part and spread-out part, we have
\[ \int_{\R^2} (\sum_\ga \Id_{\ga^*_r})^{p/2}\approx \int_{B(0,1)}(\sum_\ga \Id_{\ga^*_r})^{p/2} + \int_{B(0,R^\beta)\setminus B(0,\frac{1}{2}R^\beta)}(\sum_\ga \Id_{\ga^*_r})^{p/2}\sim R^{\frac{p\beta}{2}}+R^{2\beta}. \]

\medskip

\noindent\fbox{Case 2: $10\le r\le R^\beta$}  

For any point $P\in\bigcup \ga^*_r$, we are going to estimate $\sum_\ga \Id_{\ga^*_r}(P)$. Define
\[ d(P):=\dist(P,S_r). \]
We see that any $P\in\bigcup \ga^*_r$ satisfies $d(P)\lesssim R^\beta$, and if $P\in\bigcup \ga^*_r$ lies inside $S_r$ then $d(P)=0$. For simplicity, we write $d=d(P)$. We consider several cases:

\begin{figure}
    \centering
    \begin{tikzpicture}[x=0.75pt,y=0.75pt,yscale=-1,xscale=1]

\draw  (74,136.75) -- (293,136.75)(228,11.75) -- (228,255.75) (286,131.75) -- (293,136.75) -- (286,141.75) (223,18.75) -- (228,11.75) -- (233,18.75)  ;
\draw   (90,136.75) .. controls (90,98.64) and (120.89,67.75) .. (159,67.75) .. controls (197.11,67.75) and (228,98.64) .. (228,136.75) .. controls (228,174.86) and (197.11,205.75) .. (159,205.75) .. controls (120.89,205.75) and (90,174.86) .. (90,136.75) -- cycle ;
\draw   (223.44,88.13) -- (232.56,88.13) -- (232.56,185.38) -- (223.44,185.38) -- cycle ;
\draw   (212.67,85.09) -- (221.47,82.66) -- (247.33,176.41) -- (238.53,178.84) -- cycle ;

\draw (292,94) node   [align=left] {\begin{minipage}[lt]{68pt}\setlength\topsep{0pt}
P
\end{minipage}};

\end{tikzpicture}

    \caption{}
    \label{fig5}
\end{figure}

\begin{enumerate}
    \item\label{first} $d\le 10$. In this case, $P$ lies in the $10$-neighborhood of $S_r$. Therefore,
    \[ \sum_\ga \Id_{\ga^*_r}(P)=\sum_\ga \Id_{\ga^*_r\cap N_{10}(S_r)}(P) \]
    Noting that $\ga^*_r\cap N_{10}(S_r)$ is essentially a $1\times r^{1/2}$-rectangle centered at $C(\ga^*_r)(\in S_r)$ and noting that $\{C(\ga^*_r)\}$ are $rR^{-\beta}$ separated, we have
    \[ \sum_\ga\Id_{\ga^*_r\cap N_{10}(S_r)}(P)\sim \frac{r^{1/2}}{rR^{-\beta}}=r^{-1/2}R^\beta. \]

    \item\label{second} $10\le d\le r$. We claim in this case
    \[ \sum_\ga \Id_{\ga^*_r}(P)\sim R^\beta (rd)^{-1/2}. \]
    See Figure \ref{fig5}. By translation and rotation, we may assume $S_r$ is centered at $(-r,0)$ and $P$ lies on the $x_2$-axis. By Pythagorean theorem, the coordinate of $P$ is $(0, \sqrt{d(d+2r)})$. Since $d\le r$, we may ignore some constant factor and write the coordinate of $P$ as 
    \begin{equation}\label{P}
        P=(0,(dr)^{1/2}).
    \end{equation} 
    The next step is to find the number of $\ga^*_r$ that pass through $P$. Suppose $P\in\ga^*_r$. Since the center of $\ga^*_r$ lies in $ S_r$, we may denote its coordinate by
    $C(\ga^*_r)=(-r+r\cos\theta,r\sin\theta)$. Let $\ell$ be the line passing through $C(\ga^*_r)$ and tangent to $S_r$ (which is also the core line of $\ga^*_r$):
    \[ \ell: y-r\sin\theta=-\frac{\cos\theta}{\sin\theta}(x+r-r\cos\theta). \]
    Since $(dr)^{1/2}\le R^\beta$, we see that $P\in \ga^*_r$ is equivalent to $\dist(\ell,P)\le \frac{1}{2}$. By some calculation, 
    \begin{align*}
        \dist(\ell,P)&=\frac{|(dr)^{1/2}-r\sin\theta+\frac{\cos\theta}{\sin\theta}r(1-\cos\theta)|}{\sqrt{1+\frac{\cos^2\theta}{\sin^2\theta}}}\\
        &=|\sin\theta(dr)^{1/2}-r(1-\cos\theta)|\\
        &=2|(dr)^{1/2}\sin\frac{\theta}{2}\cos\frac{\theta}{2}-r\sin^2\frac{\theta}{2}|.
    \end{align*} 
We just need to find the number of $\theta$
 such that $\dist(\ell,P)\le 1/2$. By symmetry, we just compute the positive solutions $\theta$ that are close to $0$. In this case, the inequality becomes
 \[ (dr)^{1/2}\sin\frac{\theta}{2}\cos\frac{\theta}{2}-r\sin^2\frac{\theta}{2}\le 1/4. \]
 The meaningful solutions will be
 \begin{align*}
     \sin\frac{\theta}{2}&\le \frac{(dr)^{1/2}\cos\frac{\theta}{2}-\sqrt{dr\cos^2\frac{\theta}{2}-r}}{2r}\\
     &=\frac{1}{2}\frac{1}{(dr)^{1/2}\cos\frac{\theta}{2}+\sqrt{dr\cos^2\frac{\theta}{2}-r}}\\
     &\sim (dr)^{-1/2}. 
 \end{align*}
In the last step, we use $\cos\frac{\theta}{2}\sim1$.
Therefore, $0\le \theta\lesim (dr)^{-1/2}$. Since $\{C(\ga^*_r)\}$ have angle separation $\sim R^{-\beta}$, we see the number of $\ga^*_r$ that contains $P$ is $\sim R^\beta(dr)^{-1/2}$.
\medskip

\item\label{third} $r\le d\le R^{\beta}$. We claim in this case
\[ \sum_\ga\Id_{\ga^*_r}(P)\sim R^\beta d^{-1}. \]
The calculation is exactly the same as above, with the only modification that we replace \eqref{P} by $P=(0,d)$.

\end{enumerate}

\medskip

Combining the three scenarios \eqref{first}, \eqref{second}, \eqref{third}, we can estimate
\begin{align*}
    \int_{\R^2} (\sum_\ga \Id_{\ga^*_r})^{p/2}&=\bigg(\int_{d(P)\le 10}+ \int_{10\le d(P)\le r}+\int_{r\le d(P)\le R^\beta} \bigg)(\sum_\ga \Id_{\ga^*_r}(P))^{p/2}dP\\
    &\sim r (r^{-1/2}R^\beta)^{p/2}+ \sum_{d\in[10,r]\textup{~dyadic}}dr (R^\beta(rd)^{-1/2})^{p/2}+\sum_{d\in[r,R^\beta]\textup{~dyadic}} d^2 (R^\beta d^{-1})^{p/2}\\
    &\approx r^{1-\frac{p}{4}}R^{\beta\frac{p}{2}}+r^{2-\frac{p}{2}}R^{\beta\frac{p}{2}}+R^{2\beta} .
\end{align*}
In the last line, we use $\approx$ is because when $p=4$, the summation is over $\sim \log R$ same numbers instead of a geometric series.

\medskip

\noindent\fbox{Case 3: $R^\beta\le r\le R$}

This is almost the same as \fbox{Case 2}. Actually, it is even simpler, since we only have scenarios \eqref{first} and \eqref{second} (with the range in \eqref{second} replaced by $10\le d\le r^{-1}R^{2\beta}$ and noting $r^{-1}R^{2\beta}\le r$). The same argument will give
\begin{align*}
    \int_{\R^2} (\sum_\ga \Id_{\ga^*_r})^{p/2}&=\bigg(\int_{d(P)\le 10}+ \int_{10\le d(P)\le r^{-1}R^{2\beta}}\bigg)(\sum_\ga \Id_{\ga^*_r}(P))^{p/2}dP\\
    &\sim r (r^{-1/2}R^\beta)^{p/2}+ \sum_{d\in[10,r^{-1}R^{2\beta}]\textup{~dyadic}}dr (R^\beta(rd)^{-1/2})^{p/2}\\
    &\approx r^{1-\frac{p}{4}}R^{\beta\frac{p}{2}}+R^{2\beta}.
\end{align*}

\end{proof}

With \eqref{slice}, we can finally estimate
\begin{align*}
    \int(\sum_\ga \Id_{\ga^*})^{p/2}&=\int_{|r|\le R}dr\int_{\{x_3=r\}}(\sum_\ga \Id_{\ga^*})^{p/2}dx_1dx_2\\
    & \bigg(\int_{0\le |r|\le 10}+\int_{10\le |r|\le R^\beta}+\int_{R^{\beta}\le |r|\le R}\bigg)(\sum_\ga \Id_{\ga^*})^{p/2}dx_1dx_2dr\\
    &\lesim R^{2\beta}+R^{\frac{p\beta}{2}}+ \sum_{r\in[10,R^\beta] \textup{~dyadic}} r(r^{1-\frac{p}{4}}R^{\beta\frac{p}{2}}+r^{2-\frac{p}{2}}R^{\beta\frac{p}{2}}+R^{2\beta})\\
    &\ \ \ \ +\sum_{r\in[R^\beta,R] \textup{~dyadic}} r(r^{1-\frac{p}{4}}R^{\beta\frac{p}{2}}+R^{2\beta})\\
    &\lesim R^{\frac{p\beta}{2}}+ 
R^{\beta(2+\frac{p}{4})}+ R^{(2-\frac{p}{4})+\frac{p\beta}{2}}+R^{1+2\beta}\\
    &\sim R^{\frac{p\beta}{2}}+R^{(2-\frac{p}{4})+\frac{p\beta}{2}}+R^{1+2\beta}.
\end{align*}
The last step is because of $R^{\beta(2+\frac{p}{4})}\le R^{\frac{p\beta}{2}}+R^{(2-\frac{p}{4})+\frac{p\beta}{2}}$.

Combining \eqref{lowbound}, \eqref{upbound} and plugging into \eqref{ineqvcone}, we obtain
\[ \max\{R^{-1},R^{-\frac{3}{2}+\frac{2}{p}}\}\lessapprox C_{\beta,p}(R) R^{-1-\beta} (R^{\frac{\beta}{2}}+R^{\frac{2}{p}-\frac{1}{4}+\frac{\beta}{2}}+R^{\frac{1+2\beta}{p}}). \]
Considering of the three cases $2\le p\le 4$, $4\le p\le 8$ and $p\ge 8$ will give us that the right hand side of \eqref{cbeta} is actually the lower bound of $C_{\beta,p}(R)$ (up to $R^\e$ factor).

\subsection{Proof of Theorem \ref{thmvcone}}
The difficult part of the proof will be in the range $4\le p\le 8$. Recall from Remark \ref{rmk} that we need to prove for all $p$ but not only the endpoint $p$, since there is no interpolation argument.
The main tool we are going to use is called the \textit{amplitude dependent wave envelope estimate} by Guth-Maldague \cite{guth2022amplitude}. Before giving the proof, we introduce some notations from \cite{guth2020sharp}, \cite{guth2022amplitude}.

Recall $\cC$ is the truncated cone in $\R^3$:
$$\cC:=\{\xi\in\R^3:\xi_3=\sqrt{x_1^2+x_2^2},1/2\le \xi_3\le 1\}.$$
We have the canonical partition of $N_{R^{-1}}(\cC)$ into $1\times R^{-1/2}\times R^{-1}$-planks $\Theta=\{\theta\}$: 
$$N_{R^{-1}}(\cC)=\bigsqcup \theta.$$
More generally, for any dyadic $s\in [R^{-1/2},1]$, we can partition the $s^2$-neighborhood of $\cC$ into $1\times s\times s^2$-planks $\bS_s=\{\tau_s\}$: $$N_{s^2}(\cC)=\bigsqcup \tau_s.$$
Note in particular $\bS_{R^{-1/2}}=\Theta$.
For each $s$ and a frequency plank $\tau_s\in \bS_s$, we define the box $U_{\tau_s}$ in the physical space to be a rectangle centered at the origin of dimensions $Rs^2\times Rs\times R$ whose edge of length $Rs^2$ (respectively $Rs$, $R$) is parallel to the edge of $\tau_s$ with length $1$ (respectively $s$, $s^2$). Note that for any $\theta\in\Theta$, $U_\theta$ is just 
 $\theta^*$ (the dual rectangle of $\theta$). Also, $U_{\tau_s}$ is the convex hull of $\cup_{\theta\subset \tau_s}U_{\theta}$.

We make a useful observation, which will be used later. For any $\theta\subset \tau_s$, we see that $\theta^*$ is a $1\times R^{1/2}\times R$-plank. Define $U_{\theta,s}$ to be the $Rs^2\times Rs\times R$-plank which is made by dilating the corresponding edges of $\theta^*$. Our observation is that $U_{\tau_s}$ and $U_{\theta,s}$ are comparable:
\begin{equation}\label{compa}
    \frac{1}{C}  U_{\theta,s}\subset U_{\tau_s}\subset C  U_{\theta,s}.
\end{equation}
This is not hard to see by noting that the second longest edge of $\theta^*$ form an angle $\lesim s$ with the $Rs\times R$-face of $U_{\tau_s}$. We just omit the proof.

We cover $\R^3$ by translated copies of $U_{\tau_s}$. We will use $U\parallel U_{\tau_s}$ to indicate $U$ is one of the translated copies.
If $U\parallel U_{\tau_s}$, then we define $S_U f$ by
\begin{equation}\label{defSf}
    S_U f=\big(\sum_{\theta\subset \tau_s}|f_\theta|^2\big)^{1/2}\Id_U.
\end{equation}
We can think of $S_U f$ as the wave envelope of $f$ localized in $U$ in the physical space and localized in $\tau_s$ in the frequency space.
We have the following inequality of Guth, Wang and Zhang (see \cite[Theorem 1.5]{guth2020sharp} ):
\begin{theorem}[Wave envelope estimate]\label{GWZstandard}
Suppose $\supp\wh f\subset N_{R^{-1}}(\cC)$. Then
\begin{equation}\label{standardineq}
    \|f\|_{4}^4\le C_\e R^{\e} \sum_{R^{-1/2}\le s\le 1}\sum_{\tau_s\in\bS_s}\sum_{U\parallel U_{\tau_s}} |U|^{-1}\|S_Uf\|_2^4,
\end{equation}
for any $\e>0$.
\end{theorem}

There is a refined version of the wave envelope estimate proved by Guth and Maldague (See \cite[Theorem 2]{guth2022amplitude} ):
\begin{theorem}[Amplitude dependent wave envelope estimate]\label{amplitudethm}
    Suppose $\supp\wh f\subset N_{R^{-1}}(\cC)$. Then for any $\alpha>0$,
\begin{equation}\label{amplitudeineq}
    \alpha^4 |\{ x\in\R^3:|f(x)|>\alpha \}|\le C_\e R^{\e} \sum_{R^{-1/2}\le s\le 1}\sum_{\tau_s\in\bS_s}\sum_{U\in \cG_{\tau_s}(\alpha)} |U|^{-1}\|S_Uf\|_2^4, 
\end{equation}
for any $\e>0$.
Here, $\cG_{\tau_s}(\alpha)=\{ U\parallel U_{\tau_s}:|U|^{-1}\|S_U f\|_2^2\gtrsim |\log R|^{-1} \frac{\alpha^2}{(\#\bS_s)^2} \}$.
\end{theorem}

\begin{remark}
{\rm
    In the original paper \cite{guth2022amplitude}, their definition for $\cG_{\tau_s}(\alpha)$ is 
    \[\cG_{\tau_s}(\alpha)=\{ U\parallel U_{\tau_s}:|U|^{-1}\|S_U f\|_2^2\gtrapprox \frac{\alpha^2}{(\#\tau_s)^2} \},\] 
    where $\#\tau_s=\#\{\tau_s\in\bS_s:f_{\tau_s}\not\equiv 0\}$. Noting that $\#\tau_s\le \#\bS_s$, we see our $\cG_{\tau_s}(\alpha)$ is a bigger set, and hence our \eqref{amplitudeineq} is weaker than the original version (\cite{guth2022amplitude} Theorem 2). 
}
\end{remark}

\begin{proof}[Proof of Theorem \ref{thmvcone}]\hfill

\medskip
\noindent\fbox{Case 1: $p\ge 8$}

This is just by Cauchy-Schwarz inequality, since $\#\Ga_\beta(R^{-1})\sim R^{\beta}.$

\medskip

\noindent\fbox{Case 2: $2\le p\le 4$}

We have \eqref{standardineq}. By dyadic pigeonholing on $s$, we can find $s$ such that
\begin{equation}
    \|f\|_4^4\lessapprox \sum_{\tau\in\bS_s}\sum_{U\parallel U_\tau} |U|^{-1}\|S_U f\|_2^4.
\end{equation}
We fix this $s$.
Denote $\bU:=\{U: U\parallel U_\tau \textup{~for~some~}\tau\in\bS_s\}$. Then the inequality above can be written as

\begin{equation}\label{L4}
    \|f\|_4^4\lessapprox \sum_{U\in\bU} |U|^{-1}\|S_U f\|_2^4.
\end{equation}
We remind readers that each $U\in\bU$ has size $Rs^2\times Rs\times R$. We also have the following $L^2$ estimate:
\begin{equation}\label{L2}
    \|f\|_2^2\sim \sum_{U\in \bU} \|S_U f\|^2_2.
\end{equation}
We provide a quick proof for \eqref{L2}. We have
\[ \|f\|_2^2=\sum_{\tau\in \bS_s}\|f_\tau\|_2^2\sim\sum_{\tau\in \bS_s}\sum_{U\parallel U_\tau}\|f_\tau\|_{L^2(U)}^2. \]
Noting that $\{f_\theta: \theta\subset \tau\}$ are locally orthogonal on any translation of $U_\tau$ and recalling \eqref{defSf}, we have
\[ \|f\|_2^2\sim \sum_{\tau\in\bS_s}\sum_{U\parallel U_\tau}\int_U \sum_{\theta\subset\tau}|f_\theta|^2=\sum_{U\in\bU}\|S_U f\|_2^2. \]

Next, we will do dyadic pigeonholing on $\|S_U f\|_2^2$. (Actually, we only need to prove a local version of the inequality, so we just care about those $U$ that intersect $B_R$. There are in total $R^{O(1)}$ of them.) We can find a number $W>0$ and set $\bU'=\{U\in\bU: \|S_Uf\|_2^2\sim W\}$, so that
\begin{align}
    \label{LL4}\|f\|_4^4&\lessapprox |U|^{-1}\#\bU' W^2,\\
    \label{LL2}\|f\|_2^2&\approx \#\bU' W.
\end{align}
Since every $U\in\bU$ has the same measure $R^3s^2$, there is no ambiguity to write $|U|^{-1}$ in \eqref{LL4}.

Let $\alpha$ be such that $\frac{1}{p}=\frac{\alpha}{4}+\frac{1-\alpha}{2}$. Then $\alpha=4(\frac{1}{2}-\frac{1}{p})$. Applying H\"older's inequality gives
\[ \|f\|_p^p\le \|f\|_4^{\alpha p} \|f\|_2^{(1-\alpha)p}\lessapprox |U|^{-p(\frac{1}{2}-\frac{1}{p})}\#\bU' W^{\frac{p}{2}}\le |U|^{-p(\frac{1}{2}-\frac{1}{p})}\sum_{U\in\bU} \|S_U f\|_2^p. \]

Next we are going to exploit more orthogonality for $S_U f$. Suppose $U\parallel U_\tau$. By definition
\[ \|S_U f\|_2^2=\int_U \sum_{\theta\subset \tau}|f_\theta|^2=\int_U \sum_{\theta\subset \tau}|\sum_{\ga\subset \theta}f_\ga|^2. \]
We remind readers that $\{\tau\}$ are $1\times s\times s^2$-caps; $\{\theta\}$ are $1\times R^{-1/2}\times R^{-1}$-caps; $\{\ga\}$ are $1\times R^{-\beta}\times R^{-1}$-caps.
Since $U$ is too small for $\{f_\ga:\ga\subset \theta\}$ to be orthogonal on $U$, we need to find a larger rectangle. First, let us look at the rectangles $\{\ga: \ga\subset \theta\}$. We want to find a rectangle $\nu_\theta$ as big as possible, such that $\{\ga+\nu: \ga\subset \theta\}$ are finitely overlapping. Actually, we can choose $\nu_\theta$ to be of size $R^{1/2-\beta}\times R^{-\beta}\times R^{-1}$ (here the edge of $\nu_\theta$ with length $R^{1/2-\beta}$ (respectively $R^{-\beta}$, $R^{-1}$) are parallel to the edge of $\theta$ with length $1$ (respectively $R^{-1/2}$, $R^{-1}$). See Figure \ref{fig5.5}: the left hand side is $\theta$ and $\{\ga:\ga\subset\theta\}$; the right hand side is our $\nu_\theta$. It is not hard to see $\{\ga+\nu_\theta: \ga\subset \theta\}$ are finitely overlapping. Let $\nu_\theta^*$ be the dual of $\nu_\theta$ in the physical space, then $\nu_\theta^*$ has size $R^{\beta-\frac{1}{2}}\times R^\be\times R$ and we have the local orthogonality (we just ignore the rapidly decaying tail for simplicity):
\[ \int_{\nu_\theta^*} |\sum_{\ga\subset\theta} f_\ga|^2 \sim \int_{\nu_\theta^*} \sum_{\ga\subset\theta} |f_\ga|^2 \]

\begin{figure}
    \centering
    \begin{tikzpicture}[x=0.75pt,y=0.75pt,yscale=-1,xscale=1]

\draw   (240.12,50.8) -- (221.82,253.7) -- (134.63,253.61) -- (116.75,50.67) -- cycle ;
\draw    (141,51) -- (155,253.5) ;
\draw    (164,51.25) -- (171,253.75) ;
\draw    (214.5,51.25) -- (205,254.5) ;
\draw    (188,50.75) -- (188,253.5) ;
\draw   (356.5,179.25) -- (376.5,179.25) -- (376.5,252) -- (356.5,252) -- cycle ;

\draw (173,11.4) node [anchor=north west][inner sep=0.75pt]    {$\theta $};
\draw (131,137.4) node [anchor=north west][inner sep=0.75pt]    {$\gamma $};
\draw (162,263.4) node [anchor=north west][inner sep=0.75pt]    {$R^{-1/2}$};
\draw (112,23.4) node [anchor=north west][inner sep=0.75pt]    {$R^{-\beta }$};
\draw (356.5,198.4) node [anchor=north west][inner sep=0.75pt]    {$\nu _{\theta }$};
\draw (385,209.9) node [anchor=north west][inner sep=0.75pt]    {$R^{1/2-\beta }$};
\draw (351.5,153.4) node [anchor=north west][inner sep=0.75pt]    {$R^{-\beta }$};

\end{tikzpicture}

    \caption{}
    \label{fig5.5}
\end{figure}

Define 
\begin{equation}\label{defV}
    V_\theta=U_\tau+\nu_\theta^*,
\end{equation}
which is a rectangle of size 
\[\max\{Rs^2,R^{\beta-\frac{1}{2}}\}\times \max\{Rs,R^\beta\}\times R.\]
We tile $\R^3$ with translated copies of $V_\theta$, and we write $V\parallel V_\theta$ if $V$ is one of the tiles. 
Noting that $\frac{R^{\beta-\frac{1}{2}}}{Rs^2}\le \frac{R^\beta}{Rs}$, we will discuss three scenarios: 1. $\frac{R^{\beta-\frac{1}{2}}}{Rs^2}\le \frac{R^\beta}{Rs}\le 1$; 2. $\frac{R^{\beta-\frac{1}{2}}}{Rs^2}\le 1\le \frac{R^\beta}{Rs}$; 3. $1\le\frac{R^{\beta-\frac{1}{2}}}{Rs^2}\le \frac{R^\beta}{Rs}$. 

$\bullet$ If $\frac{R^{\beta-\frac{1}{2}}}{Rs^2}\le \frac{R^\beta}{Rs}\le 1$, then $V_\theta$ is essentially $U_\tau$. In this case, we already have the orthogonality of $\{f_\ga: \ga\subset\theta\}$ on $U(\parallel U_\tau)$. Therefore,
\begin{align*}
    \|f\|_p^p&\lessapprox |U|^{-p(\frac{1}{2}-\frac{1}{p})}\sum_{\tau\in\bS_s}\sum_{U\parallel U_\tau}\bigg(\int_U \sum_{\theta\subset \tau}|\sum_{\ga\subset \theta}f_\ga|^2\bigg)^{\frac{p}{2}}\\
    &\sim|U|^{-p(\frac{1}{2}-\frac{1}{p})}\sum_{\tau\in\bS_s}\sum_{U\parallel U_\tau}\bigg(\int_U \sum_{\ga\subset\tau}|f_\ga|^2\bigg)^{\frac{p}{2}}\\
    &\le \sum_{\tau\in\bS_s}\sum_{U\parallel U_\tau}\int_U(\sum_{\ga\subset\tau}|f_\ga|^2)^{\frac{p}{2}}\\
    &=\sum_{\tau\in\bS_s}\int_{\R^3}(\sum_{\ga\subset\tau}|f_\ga|^2)^{\frac{p}{2}}\\
    &\le \int_{\R^3}(\sum_{\ga\in\Ga_{\beta}(R^{-1})}|f_\ga|^2)^{p/2}.
\end{align*} 

$\bullet$ In the other two scenarios, we proceed as follows.
\begin{align*}
    \|f\|_p^p&\lessapprox|U|^{-p(\frac{1}{2}-\frac{1}{p})}\sum_{\tau\in\bS_s} \sum_{U\parallel U_\tau} \bigg(\int_U \sum_{\theta\subset\tau}|f_\theta|^2\bigg)^{p/2}\\
    &\le |U|^{-p(\frac{1}{2}-\frac{1}{p})}\sum_{\tau\in\bS_s} \sum_{U\parallel U_\tau} \#\{\theta\subset\tau\}^{\frac{p}{2}-1}\sum_{\theta\subset\tau}\bigg(\int_U |f_\theta|^2\bigg)^{p/2}\\
    &\le |U|^{-p(\frac{1}{2}-\frac{1}{p})}\#\{\theta\subset\tau\}^{\frac{p}{2}-1}\sum_{\tau\in\bS_s}  \sum_{\theta\subset\tau} 
\sum_{V\parallel V_\theta} \bigg(\int_V |f_\theta|^2\bigg)^{p/2}\\
(\textup{By~orthogonality})&\sim |U|^{-p(\frac{1}{2}-\frac{1}{p})}\#\{\theta\subset\tau\}^{\frac{p}{2}-1}\sum_{\tau\in\bS_s}  \sum_{\theta\subset\tau} 
\sum_{V\parallel V_\theta} \bigg(\int_V \sum_{\ga\subset\theta}|f_\ga|^2\bigg)^{p/2}\\
(\textup{H\"older})&\le |U|^{-p(\frac{1}{2}-\frac{1}{p})}\#\{\theta\subset\tau\}^{\frac{p}{2}-1}\sum_{\tau\in\bS_s}  \sum_{\theta\subset\tau} 
\sum_{V\parallel V_\theta} |V|^{p(\frac{1}{2}-\frac{1}{p})}\int_V (\sum_{\ga\subset\theta}|f_\ga|^2)^{p/2}\\
&\le \bigg(\frac{|V|}{|U|}\bigg)^{p(\frac{1}{2}-\frac{1}{p})}\#\{\theta\subset\tau\}^{\frac{p}{2}-1} \|(\sum_{\ga\in\Ga_\beta(R^{-1})}|f_\ga|^2)^{\frac{1}{2}}\|_p^p\\
&= \bigg(\max\{\frac{R^\beta}{Rs},1\}\max\{\frac{R^{\beta-\frac{1}{2}}}{Rs^2},1\}\bigg)^{p(\frac{1}{2}-\frac{1}{p})} (s R^{\frac{1}{2}})^{\frac{p}{2}-1}\|(\sum_{\ga\in\Ga_\beta(R^{-1})}|f_\ga|^2)^{\frac{1}{2}}\|_p^p
\end{align*} 

We just need to check 
\begin{equation}\label{tocheck}
    \bigg(\max\{\frac{R^\beta}{Rs},1\}\max\{\frac{R^{\beta-\frac{1}{2}}}{Rs^2},1\}\bigg)^{p(\frac{1}{2}-\frac{1}{p})} (s R^{\frac{1}{2}})^{\frac{p}{2}-1}\lesim R^{(\beta-\frac{1}{2})(p-2)}. 
\end{equation} 
$*$ If $\frac{R^{\beta-\frac{1}{2}}}{Rs^2}\le 1\le \frac{R^\beta}{Rs}$, then the left hand side of \eqref{tocheck} equals $R^{(\beta-\frac{1}{2})(\frac{p}{2}-1)}$, which is $\le$ the right hand side of \eqref{tocheck}.

\noindent
$*$ If $1\le\frac{R^{\beta-\frac{1}{2}}}{Rs^2}\le \frac{R^\beta}{Rs}$, then the left hand side of \eqref{tocheck} equals
\[ (R^{2\beta-2}s^{-2})^{\frac{p}{2}-1}, \]
which is less than the right hand side of \eqref{tocheck} since $s^{-1}\le R^{1/2}$.
\medskip

\noindent\fbox{Case 3: $4\le p\le 8$}

Note that 
\[ \|f\|_p^p\sim \sum_{\alpha \text{~dyadic}}\alpha^p |\{x\in\R^3:|f(x)|\sim \alpha\}|. \]
We can assume the range of $\alpha$ is $R^{-100}\|f\|_\infty\le \alpha \le \|f\|_\infty$. Other $\alpha$ are considered as negligible.

By dyadic pigeonholing, we can find $\alpha>0$ such that
\[ \|f\|_p^p\lesim (\log R)\cdot\alpha^p |\{x\in\R^3:|f(x)|\sim \alpha\}|+\textup{negligible~term}. \]
We just need to fix this $\alpha$, and prove an upper bound for $\alpha^p |\{x\in\R^3:|f(x)|> \alpha\}|$.
By \eqref{amplitudeineq}, we have
\[ \alpha^4 |\{ x\in\R^3:|f(x)|>\alpha \}|\le C_\e R^{\e} \sum_{R^{-1/2}\le s\le 1}\sum_{\tau_s\in\bS_s}\sum_{U\in \cG_{\tau_s}(\alpha)} |U|^{-1}\|S_Uf\|_2^4.\]
By pigeonholing again, we can find $s$ such that
\begin{equation}\label{amp}
    \alpha^4 |\{ x\in\R^3:|f(x)|>\alpha \}|\lessapprox\sum_{\tau\in\bS_s}\sum_{U\in \cG_{\tau}(\alpha)} |U|^{-1}\|S_Uf\|_2^4.
\end{equation}
We fix this $s$. We also remind readers the definition of $\cG_\tau(\alpha)$:
\[  \cG_\tau(\alpha):=\{ U\parallel U_\tau: |U|^{-1}\int_U\sum_{\theta\subset \tau}|f_\theta|^2\gtrapprox (\alpha s)^2 \}, \]
since $\#\bS_s\sim s^{-1}$.
Continuing the estimate in \eqref{amp}, we have
\begin{align*}
    \alpha^4 |\{ x\in\R^3:|f(x)|>\alpha \}|&\lessapprox\sum_{\tau\in\bS_s}\sum_{U\in \cG_{\tau}(\alpha)} |U|^{-1}\bigg(\int_{U} \sum_{\theta\subset\tau}|f_\theta|^2\bigg)^2\\
    &\lessapprox \sum_{\tau\in\bS_s}\sum_{U\in \cG_{\tau}(\alpha)} |U|^{-1}\bigg(\int_{U} \sum_{\theta\subset\tau}|f_\theta|^2\bigg)^{\frac{p}{2}}\bigg( |U|(\alpha s)^2 \bigg)^{2-\frac{p}{2}}.
\end{align*}
Moving the power of $\alpha$ to the left hand side, we obtain
\begin{equation}
    \alpha^p |\{ x\in\R^3:|f(x)|>\alpha \}|\lessapprox\sum_{\tau\in\bS_s}\sum_{U\in \cG_{\tau}(\alpha)} |U|^{1-\frac{p}{2}}\bigg(\int_{U} \sum_{\theta\subset\tau}|f_\theta|^2\bigg)^{\frac{p}{2}}s^{4-p}.
\end{equation}
Our final goal is to prove that the right hand side above is 
\begin{equation}\label{finalgoal}
    \lessapprox R^{\frac{\beta p}{2}+\frac{p}{4}-2}\|(\sum_\ga |f_\ga|^2)^{1/2}\|_p^p.
\end{equation}

To do that, we again need to exploit the orthogonality of $\{f_\ga:\ga\subset\theta\}$. The argument is different from that in \fbox{Case 2: $2\le p\le 4$}. In \fbox{Case 2: $2\le p\le 4$}, we expand the integration domain $U$ to a bigger rectangle $V$ to get orthogonality, whereas here we are going to use Cauchy-Schwarz inequality.

We discuss the geometry of these caps. Fix a $\tau\in\bS_s$. By definition, $U_\tau$ is a $Rs^2\times Rs\times R$-rectangle in the physical space. Then $U_\tau^*$ is a $R^{-1}s^{-2}\times R^{-1}s^{-1}\times R^{-1}$-rectangle. We make the following observation: for each $\theta\subset \tau$, we can show that $U^*_\tau$ is comparable to another rectangle, which has the same size but with the edges parallel to the corresponding edges of $\theta$. We explain it with more details. Let $U_{\theta,s}$ be the $Rs^2\times Rs\times R$-rectangle which is made from the $1\times R^{1/2}\times R$-rectangle $\theta^*$ by dilating the corresponding edges. Then $U_{\theta,s}^*$ is a $R^{-1}s^{-2}\times R^{-1}s^{-1}\times R^{-1}$-rectangle whose edges are parallel to the corresponding edges of the $1\times R^{-1/2}\times R^{-1}$-rectangle $\theta$. We want to show $U_\tau^*$ and $U_{\theta,s}^*$ are comparable. This is equivalent to show $U_\tau$ and $U_{\theta,s}$ are comparable, which is an observation we made at \eqref{compa}. Therefore, for any $\theta\subset\tau$, we can assume the edges of $U_\tau^*$ are parallel to the corresponding edges of $\theta$.

Fix a $U\parallel U_\tau$,  then $U^*=U_\tau^*$.
See Figure \ref{fig6}: on the left is $\theta$ and $\{\ga: \ga\subset \theta\}$; on the middle is our $U^*$. We will discuss two scenarios depending on whether $R^{-\beta}$ (the width of $\ga$) is bigger than $R^{-1}s^{-1}$ (the width of $U^*$).

\begin{figure}
    \centering
    \begin{tikzpicture}[x=0.75pt,y=0.75pt,yscale=-1,xscale=1]

\draw   (220.12,60.8) -- (201.82,263.7) -- (114.63,263.61) -- (96.75,60.67) -- cycle ;
\draw    (121,61) -- (135,263.5) ;
\draw    (144,61.25) -- (151,263.75) ;
\draw    (194.5,61.25) -- (185,264.5) ;
\draw    (168,60.75) -- (168,263.5) ;
\draw   (289,194.5) -- (326,194.5) -- (326,260.5) -- (289,260.5) -- cycle ;
\draw   (582.12,64.8) -- (563.82,267.7) -- (476.63,267.61) -- (458.75,64.67) -- cycle ;
\draw    (496,65.5) -- (509,267.5) ;
\draw    (541,65.5) -- (535,268.5) ;

\draw (153,21.4) node [anchor=north west][inner sep=0.75pt]    {$\theta $};
\draw (111,147.4) node [anchor=north west][inner sep=0.75pt]    {$\gamma $};
\draw (142,273.4) node [anchor=north west][inner sep=0.75pt]    {$R^{-1/2}$};
\draw (92,33.4) node [anchor=north west][inner sep=0.75pt]    {$R^{-\beta }$};
\draw (335,217.4) node [anchor=north west][inner sep=0.75pt]    {$R^{-1} s^{-2}$};
\draw (284,275.4) node [anchor=north west][inner sep=0.75pt]    {$R^{-1} s^{-1}$};
\draw (297,219.4) node [anchor=north west][inner sep=0.75pt]    {$U^{\ast }$};
\draw (472,142.4) node [anchor=north west][inner sep=0.75pt]    {$\pi $};

\end{tikzpicture}

    \caption{}
    \label{fig6}
\end{figure}

$\bullet$ If $R^{-\beta}\ge R^{-1}s^{-1}$, then we see that $\{ \ga+U^*: \ga\subset \theta \}$ are finitely overlapping. This means that $\{f_\ga: \ga\subset\theta\}$ are locally orthogonal on $U$: 
\[ \int_U |\sum_{\ga\subset\theta}f_\ga|^2\lesim \int_U \sum_{\ga\subset\theta}|f_\ga|^2. \]
Therefore,
\begin{align*}
    \|f\|_p^p&\lessapprox |U|^{1-\frac{p}{2}}\sum_{\tau\in\bS_s}\sum_{U\parallel U_s}\bigg(\int_U \sum_{\theta\subset \tau}|\sum_{\ga\subset \theta}f_\ga|^2\bigg)^{\frac{p}{2}}s^{4-p}\\
    &\lesim|U|^{1-\frac{p}{2}}\sum_{\tau\in\bS_s}\sum_{U\parallel U_s}\bigg(\int_U \sum_{\ga\subset\tau}|f_\ga|^2\bigg)^{\frac{p}{2}}s^{4-p}\\
    (\textup{H\"older}) &\le s^{4-p}\sum_{\tau\in\bS_s}\sum_{U\parallel U_s}\int_U(\sum_{\ga\subset\tau}|f_\ga|^2)^{\frac{p}{2}}\\
    &=s^{4-p}\sum_{\tau\in\bS_s}\int_{\R^3}(\sum_{\ga\subset\tau}|f_\ga|^2)^{\frac{p}{2}}\\
    &\le s^{4-p}\int_{\R^3}(\sum_{\ga\in\Ga_{\beta}(R^{-1})}|f_\ga|^2)^{p/2}.
\end{align*} 
We just need to check
\[ s^{4-p}\le R^{\frac{\beta p}{2}+\frac{p}{4}-2}.\]
Plugging $s^{-1}\le R^{-1/2}$, the inequality above is reduced to
\[ R^{p/4}\le R^{\frac{\beta p}{2}},  \]
which is true since $\beta\ge 1/2$.

\medskip

$\bullet$ If $R^{-\beta}\le R^{-1}s^{-1}$, we will define a set of new planks which we call $\pi$. See on the right hand side of Figure \ref{fig6}. We partition $\theta$ into a set of $1\times R^{-1}s^{-1}\times R^{-1}$-planks, which we denoted by $\{\pi: \pi\subset \theta\}$. If the partition is well chosen (the size of caps can vary within a constant multiple), we can assume each $\ga$ fits into one $\pi$, so we define
\[ f_\pi:=\sum_{\ga\subset\pi}f_\ga. \]
Now, our key observation is that $\{\pi+U^*: \pi\subset \theta\}$ are finitely overlapping. This is true by noting that: the width of $U^*$ and $\pi$ are both $R^{-1}s^{-1}$; the angle between the longest edge of $\pi$ and $U^*$ is less than $R^{-1/2}$ and $R^{-1}s^{-2}\cdot R^{-1/2}\le R^{-1}s^{-1}$.
Therefore, we have that $\{f_\pi: \pi\subset \theta\}$ are locally orthogonal on $U$, i.e.,
\begin{equation}\label{localorth2}
    \int_U |\sum_{\pi\subset\theta}f_\pi|^2\lesim \int_U \sum_{\pi\subset\theta}|f_\pi|^2.
\end{equation}
Another step of Cauchy-Schwarz will give
\begin{equation}
    \int_U \sum_{\pi\subset\theta}|f_\pi|^2=\int_U \sum_{\pi\subset\theta}|\sum_{\ga\subset\pi}f_\ga|^2\le \#\{\ga\subset\pi\} \int_U \sum_{\ga\subset\theta}|f_\ga|^2=R^\beta R^{-1}s^{-1} \int_U \sum_{\ga\subset\theta}|f_\ga|^2.
\end{equation}
As a result, we obtain
\[ \int_U |f_\theta|^2\lesim R^\beta R^{-1}s^{-1}\int_U \sum_{\ga\subset\theta}|f_\ga|^2. \]
Summing over $\theta\subset\tau$, we obtain
\[ \int_U \sum_{\theta\subset\tau}|f_\theta|^2\lesim R^\beta R^{-1}s^{-1}\int_U \sum_{\ga\subset\tau}|f_\ga|^2. \]
Therefore,
\begin{align*}
    \|f\|_p^p&\lessapprox|U|^{1-\frac{p}{2}}\sum_{\tau\in\bS_s} \sum_{U\parallel U_\tau} \bigg(\int_U \sum_{\theta\subset\tau}|f_\theta|^2\bigg)^{p/2}s^{4-p}\\
    &\lesim |U|^{1-\frac{p}{2}}(R^\beta R^{-1}s^{-1})^{\frac{p}{2}} \sum_{\tau\in\bS_s} \sum_{U\parallel U_\tau} \bigg(\int_U \sum_{\ga\subset\tau}|f_\ga|^2\bigg)^{p/2}s^{4-p}\\
(\textup{H\"older})&\le s^{4-p}(R^\beta R^{-1}s^{-1})^{\frac{p}{2}}\sum_{\tau\in\bS_s}  \sum_{U\parallel U_\tau} \int_U (\sum_{\ga\subset\tau}|f_\ga|^2)^{p/2}\\
&\le s^{4-p}(R^\beta R^{-1}s^{-1})^{\frac{p}{2}}\|(\sum_{\ga\in\Ga_\beta(R^{-1})}|f_\ga|^2)^{\frac{1}{2}}\|_p^p.
\end{align*} 

We just need to check 
\[s^{4-p}(R^\beta R^{-1}s^{-1})^{\frac{p}{2}}\le R^{\frac{\beta p}{2}+\frac{p}{4}-2},\]
which is equivalent to 
\[ s^{4-\frac{3p}{2}}\le R^{\frac{3p}{4}-2}. \]
Plugging $s^{-1}\le R^{1/2}$ and noting that $4-\frac{3p}{2}<0$, we prove the result.


\end{proof}

\newpage

\bibliographystyle{abbrv}
\bibliography{bibli}

\end{document}